\documentclass{amsart}
\usepackage{amsmath,amssymb,amsthm}
\usepackage[mathscr]{euscript}
\usepackage{a4wide}
\usepackage{graphicx}
\usepackage{array}
\usepackage{url}
\usepackage{xypic}\xyoption{tips}
\usepackage{tikz-cd}

\newtheorem{theorem}{Theorem}
\newtheorem{lemma}[theorem]{Lemma}

\newcommand{\Res}{\mathrm{Res}}

\newcommand{\PGL}{\mathrm{PGL}}
\newcommand{\PSL}{\mathrm{PSL}}

\newcommand{\isom}{\cong}
\newcommand{\disc}{\mathrm{disc}}

\newcommand{\EE}{\mathrm{E}}
\newcommand{\FF}{\mathbb{F}}
\newcommand{\HH}{\mathbb{H}}
\newcommand{\RR}{\mathbb{R}}
\newcommand{\CC}{\mathbb{C}}
\newcommand{\NN}{\mathbb{N}}
\newcommand{\PP}{\mathbb{P}}
\newcommand{\QQ}{\mathbb{Q}}

\newcommand{\ZZ}{\mathbb{Z}}
\newcommand{\QQbar}{\overline{\QQ}}

\newcommand{\Frob}{\mathrm{Frob}}
\newcommand{\Gal}{\mathrm{Gal}}
\newcommand{\GalQQ}{\Gal(\QQbar/\QQ)}
\newcommand{\Hom}{\mathrm{Hom}}

\newcommand{\Cl}{{\mathcal{C}\!\ell}}

\newcommand{\GL}{\mathrm{GL}}
\newcommand{\rO}{\mathrm{O}}
\newcommand{\SO}{\mathrm{SO}}
\newcommand{\SU}{\mathrm{SU}}
\newcommand{\USp}{\mathrm{USp}}

\newcommand{\cC}{\mathcal{C}}

\newcommand{\cG}{\mathcal{G}}

\newcommand{\cO}{\mathcal{O}}
\newcommand{\cP}{\mathcal{P}}
\renewcommand{\cR}{\mathfrak{R}}

\newcommand{\cX}{\mathcal{X}}

\newcommand{\ee}{\mathbf{e}}
\newcommand{\mm}{\mathfrak{m}}
\newcommand{\pp}{\mathfrak{p}}
\newcommand{\fP}{\mathfrak{P}}

\newcommand{\bbmu}{\boldsymbol{\mu}}

\newcommand{\ccomma}{\raisebox{0.4ex}{,}}

\newcommand{\ignore}[1]{}
\newcommand{\longversion}[2]{#1} 

\newcommand{\Alt}[1]{{\mathcal{A}}_{#1}}
\newcommand{\Sym}[1]{{\mathcal{S}}_{#1}}

\title{Arithmetic statistics of Galois groups}
\author{David Kohel}
\address{Aix Marseille Univ, CNRS, Centrale Marseille, I2M, Marseille, France}
\email{David.Kohel@univ-amu.fr}

\begin{document}

\maketitle

\begin{abstract}
We develop a computational framework for the statistical characterization of Galois
characters with finite image, with application to characterizing Galois groups and
establishing equivalence of characters of finite images of $\GalQQ$.
\end{abstract}

\section{Introduction}

The absolute Galois group $\cG = \GalQQ$ is a fundamental object of study in number theory.
The objective of this work is to develop an explicit computational framework for the
study of its finite quotients.  We may replace $\cG$ with the absolute Galois group of any
global field, but restrict to that of $\QQ$ for simplicity of exposition.

As point of departure, we consider an irreducible polynomial $f(x) \in \ZZ[x]$ of degree $n$
as input.  We set $K = \QQ[x]/(f(x))$, denote by $L$ its normal closure and by $\cG(K)$
the Galois group $\Gal(L/\QQ)$ equipped with a permutation representation in $\Sym{n}$
determined by the action on the roots of $f(x)$.  Let $\cP_S(\ZZ)$ be the set of primes,
coprime to the finite set $S$ of primes ramified in $\ZZ[x]/(f(x))$.

The statistical perspective we develop expresses the map from $\cP_S(\ZZ)$ to factorization
data as an equidistributed map to a finite set $\cX(K)$ equipped with a probability function
induced from the Haar measure on $\cG(K)$.
A Frobenius lift at $p$, defined up to conjugacy, acts on the roots of $f(x)$.
The permutation action on the roots of $f(x)$ induces a representation in $\rO(n)$,
fixing the formal sum of the roots.
The orthogonal complement gives the standard representation in $\rO(n-1)$,
spanned by differences of basis elements.  Let $P(x)$ be the characteristic polynomial
of Frobenius in the permutation representation and
$$
S(x) = P(x)/(x-1) = x^{n-1} - s_1 x^{n-2} + \cdots + (-1)^{n-1} s_{n-1}.
$$
the characteristic polynomial in the standard representation.  This polynomial
is independent of choices of lift of Frobenius and choice of basis.
As such, the coordinates $(s_1,\dots,s_{n-1}) \in \ZZ^{n-1}$ are invariants of the Frobenius
conjugacy class $\Frob_p$ in the set $\Cl(\cG(K))$ of conjugacy classes of $\cG(K)$.
Denote the finite set of such class points by $\cX(K)$.
We note that the class points are entirely determined by the factorization data of
$f(x) \bmod p$, and $\cX(K) \subset \ZZ^{n-1}$ is equipped with the structure of
a finite probability space, induced from the cover $\Cl(\cG(K)) \rightarrow \cX(K)$.
The irreducible characters are known to form an orthogonal basis for the class functions
on $\Cl(\cG(K))$, and the rational characters are integer-valued class functions
on the class space $\cX(K)$.

In what follows we develop this approach by describing systems of rational characters on
$\cG(K)$ algebraically as a basis of polynomials in $\ZZ[s_1,\dots,s_{n-1}]$ modulo the
defining ideal for $\cX(K)$, together with their associated inner product.
As a consequence we develop algorithms for the characterization of Galois groups, and
more generally, tools for determining equivalence of finite Galois representations.

\section{Representations of orthogonal groups}

Let $G$ be a compact Lie group. In practice, $G$ will be an orthogonal group
$$
G = \rO(n-1) \subset \rO(n) \mbox{ or } G = \SO(n-1) \subset \rO(n-1),
$$
or a finite permutation group, equipped with the standard representation in $\rO(n-1)$,
$$
G \subseteq \Sym{n} \subset \rO(n-1) \mbox{ or } G \subseteq \Alt{n} \subset \SO(n-1).
$$
The standard representation of $\Sym{n}$ provides the motivation for an algebraic
presentation of the character ring of a permutation group.  For the character theory
of permutation groups, we appeal to known algorithms for its computation.

The symmetric group $\Sym{n}$ acts on a set of $n$ elements, and the linear extension to
a basis of $\ZZ^n \subset \RR^n$ gives the {\it permutation representation} of $\Sym{n}$.
Denote a basis $\{\ee_1,\dots,\ee_n\}$.  Since $\ee_1 + \cdots + \ee_n$ is fixed by $\Sym{n}$,
a line is fixed, and we consider the action on the hyperplane spanned by the orthogonal
complement.  In the basis $\{\ee_1-\ee_2,\dots,\ee_{n-1}-\ee_n\}$, we obtain the
{\it standard representation} of $\Sym{n}$ in $\rO(n-1)$.
The choice of basis is noncanonical, but the character theory is independent of any such
choice.  The orthogonal group $\rO(n)$ and its subgroup $\rO(n-1)$ have two connected
components, with principal component $\SO(n-1) \subset \SO(n)$, such that $\Alt{n} =
\Sym{n} \cap \SO(n-1)$.

\subsection*{Representation ring}

For a compact Lie group $G$, we denote the set of conjugacy classes of $G$
by $\Cl(G)$.  We define the representation ring of $G$,
$$
\cR(G) = \bigoplus_\chi \ZZ\chi,
$$
as the free abelian group on irreducible characters $\chi:G \rightarrow \CC$
of finite degree.
We identify addition with direct sum, and thereby the abelian submonoid
$\bigoplus \NN\chi \subseteq \cR(G)$ with characters, and define multiplication
on $\cR(G)$ by the linear extension of tensor product on $\bigoplus \NN\chi$.
We refer to elements of $\cR(G)$ as virtual characters.

As class functions, $\cR(G)$ can be identified with a subring of complex-valued
functions on $\Cl(G)$.  Indeed, when $G$ is finite, the number $h$ of conjugacy
classes (and of irreducible characters) is finite, and the character table is
defined as the evaluation vectors
$$
\left(\chi_i(\cC_1),\dots,\chi_i(\cC_h)\right).
$$
in the ring $\CC^h = \CC \times \cdots \times \CC$, for $\chi_i$ running over
the irreducible characters, forming a generator set for the representation ring.
For a subfield $F \subset \CC$, we denote by $\cR_F(G)$ the subring of $F$-valued
virtual characters.  While $\cR(G) = \cR_\QQ(G)$ for $G = \Sym{n}$ or
$G = \rO(n-1)$, for a general finite group that we may consider, the field
of definition of an irreducible character may be a proper extension of $\QQ$.

Considering the group $\rO(n)$ in $\GL_{n}(\RR)$, an element $g$ satisfies a
characteristic polynomial of the form
$$
x^{n} - s_1 x^{n-1} + \cdots + (-1)^{n} s_{n}.
$$
The coefficient $s_1$ is the trace in its representation on $\RR^{n}$, and
$s_{n}$ is its determinant character.  We note that $s_k$ is an invariant
of the class of $g$, and we can identify $g \mapsto s_k$ as characters.
Specifically, $s_k$ is the character on the $k$-th exterior power
$
\bigwedge\nolimits^k \RR^{n}.
$
We recall the structure of the character ring for $\rO(n)$
(cf.\ Takeuchi~\cite{Takeuchi71}).

\begin{lemma}
The virtual character ring $\cR(O(n))$ is generated by $s_k$, $1 \le k \le n$, and
$$
\cR(\rO(n)) \isom \frac{\ZZ[s_1,\dots,s_{n}]}{(s_k s_{n} - s_{n-k},s_{n}^2 - 1)}\cdot
$$
The restriction $\Res: \cR(\rO(n)) \rightarrow \cR(\SO(n))$ surjects on
$$
\cR(\SO(n)) \isom \frac{\ZZ[s_1,\dots,s_{n}]}{(s_k - s_{n-k},s_{n} - 1)}
$$
with kernel ideal $(s_n-1)$.
\end{lemma}

\noindent{\bf Remark.} If $n = 2m$ or $n = 2m+1$, then $\cR(\SO(n)) = \ZZ[s_1,\dots,s_m]$,
and $\cR(\rO(n))$ is an extension by the quadratic character $\xi = s_n$ such that
$\xi|_{\SO(n)} = 1$.

\subsection*{Algebraic parametrization}

If $H$ is a subgroup of $G$, there is an induced map $\Cl(H) \rightarrow \Cl(G)$
on conjugacy classes and concomitant restriction homomorphism $\Res: \cR(G) \rightarrow \cR(H)$
on representation rings.  Applied to the standard representation of $\Sym{n}$ in
$\rO(n-1)$, the restriction homomorphism equips the representation
ring of $\cR(\Sym{n})$ with a surjective restriction map from $\cR(\rO(n-1))$,
giving an algebraic presentation of $\cR(\Sym{n})$ by polynomials in
$\ZZ[s_1,\dots,s_{n-1}]$ modulo the defining ideal $(s_ks_{n-1}-s_{n-k-1},s^2_{n-1} - 1)$.
Given a permutation group $G \subset \Sym{n}$, the subsequent restriction
captures a significant subring of $\cR_{\QQ}(G) \subset \cR(G)$.

As a tool to characterize permutation groups in $\Sym{n}$, for subgroups $G$ and $H$,
with $H \subseteq G \subseteq \Sym{n}$, we develop the {\it branching rules} ---
explicit forms for the decomposition
$$
\Res(\chi_i) = \sum_{j=1}^{n_i} a_{ij} \psi_j.
$$
of irreducible characters $\{\chi_1,\dots,\chi_r\}$ on $G$ in terms of the irreducible
characters $\{\psi_1,\dots,\psi_s\}$ on $H$.  In light of the algebraic parametrization
by $\ZZ[s_1,\dots,s_{n-1}]$, we deduce the kernel ideals $I_G \subseteq I_H$ for each
permutation group in the lattice (poset) of subgroups.  A basis of generators provides
test functions for membership in a given subgroup.  We develop the algorithmic details
later.

Using the Brauer-Klimyk formula (see Bump~\cite[Proposition 22.9]{Bump13}),
it is possible to develop recursive formulas for the character theory of
orthogonal groups, as done in Shieh~\cite{Shieh15,Shieh16} for $\USp(2m)$, and using
the algebraic presentation, to deduce recursive branching rules for
$\Res: \cR(\rO(n-1)) \rightarrow \cR(G)$.  Instead, we content ourselves with
the algebraic parametrization from $\cR(\rO(n-1))$ and exploit the well-established
computational character theory of permutation groups to develop branching
rules in the lattice of permutation subgroups of $\Sym{n}$.

\section{Representations of permutation groups}

Let $G$ be a {\it permutation group} --- a finite group equipped with an
embedding in $\Sym{n}$.  The {\it cycle type} of $g \in G$ is the multiset
of cardinalities of its orbits under the action of $\Sym{n}$ on $\{1,\dots,n\}$.
A multiset can be denoted by a tuple $(d_1,\dots,d_t)$ or a formal product
$m_1^{e_1} \cdots m_s^{e_s}$, where
$$
d_1 \le d_2 \le \cdots \le d_t \mbox{ or }
m_1 < \cdots < m_s \mbox{ such that }
\sum_{i=1}^t d_i = \sum_{i=1}^s e_i m_i = n.
$$
The cycle type is invariant under conjugation in $\Sym{n}$, thus the cycle
type is well-defined for the conjugacy class $\cC = \cC(g) \in \Cl(G)$,
where $\cC(g) = \{ xgx^{-1} \,:\, x \in G \}$.

\begin{lemma}
\label{lem:cycle_invars}
The map $\Cl(\Sym{n}) \longrightarrow \{ (d_1,\dots,d_t) \,:\, \sum_{i=1}^t d_i = n\}$
from conjugacy classes of $\Sym{n}$ to cycle types is a bijection.
\end{lemma}

\begin{proof}
Clearly, giving a cyclic ordering to any partition of $\{1,\dots,n\}$ into
orbits determines an element of $\Sym{n}$, hence the map is surjective.
Moreover, by definition the symmetric group is $n$-transitive, conjugating
any cyclically ordered orbit partition to any another of the same cycle type.
Consequently the map is injective.
\end{proof}

\noindent{\bf Remark.}
For a permutation group $G \subset \Sym{n}$ the induced map $\Cl(G) \rightarrow
\Cl(\Sym{n})$ in general is neither injective nor surjective.  The failure of
injectivity means that the cycle type fails to distinguish the conjugacy classes.
We will later see this in the failure of $\cR(\Sym{n})$ to surject on $\cR(G)$.
In fact, the irreducible characters are known to form a basis of the class
functions on $G$ (cf. Serre~\cite[Theorem~6]{Serre71}), hence the failure
to separate conjugacy classes means that the restriction homomorphism from
$\cR(\Sym{n})$ does not surject on $\cR(G)$.
\vspace{2mm}

On the one hand, the cycle type of a conjugacy class characterizes the class.
On the other hand, the characteristic polynomial (hence its coefficients) is
class invariant of an orthogonal group element, and the permutation and
standard representations thus provide other class invariants.  We make this
association explicit.  Given $(d_1,\dots,d_t)$ be the cycle type of an
element $g \in \Sym{n}$.  It is easy to see that the characteristic polynomial
of the permutation representation of $g$ is
$$
P(x) = (x^{d_1}-1)\cdots(x^{d_t}-1) = (x^{m_1}-1)^{e_1} \cdots (x^{m_s}-1)^{e_s}.
$$
The eigenvalue on the trivial space is $1$, so the characteristic polynomial
in the standard representation is
$$
S(x) = \frac{P(x)}{(x-1)} = x^{n-1} - s_1 x^{n-2} + \cdots + (-1)^{n-1} s_{n-1},
$$
and $(s_1,\dots,s_{n-1})$ is the tuple of class invariants associated to the
the conjugacy class $\cC(g)$ under the standard representation in $\rO(n-1)$.
This gives the following lemma.

\begin{lemma}
\label{lem:symmetric_invars}
The map $\Cl(\Sym{n}) \rightarrow \ZZ^{n-1}$ from conjugacy classes to the
$(n-1)$-tuples $(s_1,\dots,s_{n-1})$ of coefficients of the characteristic
polynomial under the standard embedding is injective.
\end{lemma}

\begin{proof}
By Lemma~\ref{lem:cycle_invars} the map from conjugacy classes to cycle types
is a bijection. However, by unique factorization in $\QQ[x]$, a polynomial
of the form $(x^{d_1}-1) \cdots (x^{d_t}-1)$ is uniquely determined by the
cycle type $(d_1,\dots,d_t)$, hence the map to its coefficients $(s_1,\dots,s_{n-1})$
is injective.
\end{proof}

\subsection*{Representation rings and character tables}

Let $G$ be a permutation group and let $\Cl(G) = \{\cC_1, \dots, \cC_h\}$,
and $\{\chi_1,\dots,\chi_h\}$ be its irreducible characters.
For a conjugacy class $\cC$, define the ideal
$$
\mm_\cC = \{ f \in \cR(G) \,:\, f(\cC) = 0 \}.
$$
such that the value $f(\cC)$ of a virtual character $f$ at $\cC$ is a
well-defined class in the residue class ring $\cR(G)/\mm_\cC$.
The character table of $G$ is typically represented as a matrix whose
$i$-th row is the evaluation vector
$
(\chi_i(\cC_1),\dots,\chi_i(\cC_h)).
$
With this notation, we interpret as the embedding of the character $\chi_i$ 
in the product ring, under the injection
$$
\cR(G) \longrightarrow \cR(G)/\mm_{\cC_1} \times \cdots \times \cR(G)/\mm_{\cC_h}.
$$
\begin{lemma}
The image of the homomorphism $\cR(G) \rightarrow
\cR(G)/\mm_{\cC_1} \times \cdots \times \cR(G)/\mm_{\cC_h}$
has finite index in its codomain.
\end{lemma}

\begin{proof}
Clearly $\cR(G)$ is torsion-free, since the image of a virtual character is
a subring of $\CC$.  Thus $\cR(G)$ embeds in $\cR(G) \otimes \QQ$, which is
an \'etale algebra, and $\cR(G)/\mm_{\cC_i} \otimes \QQ$ its component fields
(see Brakenhoff~\cite{Brakenhoff} for details). It follows that the index is
finite.
\end{proof}

More generally in the direction of the Lemma, Brakenhoff~\cite{Brakenhoff} finds
that the center of the group ring $\QQ[G]$ over $\QQ$ and the tensor product of
the representation ring $\cR(G) \otimes \QQ$ are related by Brauer equivalence.
We give two examples below.  In view of the restriction map from $\cR(\Sym{n})$
to $\cR(G)$, and since all characters on $\Sym{n}$ are rational, the image of
$\cR(\Sym{n}) = \cR_\QQ(\Sym{n})$ lies in the subring $\cR_\QQ(G) \subset \cR(G)$.
In the examples below, we illustrate the role of nontrivial Galois action and
of quadratic characters in failure of surjectivity of $\cR(\Sym{n})$ on $\cR(G)$
and on $\cR_\QQ(G)$.
In the next section we exploit the embedding by interpolating the character
table values by the polynomial presentation $\ZZ[s_1,\dots,s_{n-1}]
\rightarrow \cR_\QQ(G)$.

\subsubsection*{Orthogonality relations.}

The role of arithmetic statistics of $G$ comes from the orthogonality relations
for the irreducible characters. Let $\{\chi_1,\dots,\chi_h\}$ be the irreducible
characters for $G$, and $A(G)$ be the character matrix:
$$
A(G) = \left[\begin{array}{ccc}
\chi_1(\cC_1) & \cdots & \chi_1(\cC_h) \\
\vdots        &        & \vdots \\
\chi_h(\cC_1) & \cdots & \chi_h(\cC_h)
\end{array}\right]
$$
The orthogonality relations for characters (see Serre~\cite[Section~2.3]{Serre71}),
expressed in terms of group elements, reformulated in terms of conjugacy classes,
takes the form
$$
\delta_{ij} = \langle\chi_i,\chi_j\rangle_G :=
  \frac{1}{|G|} \sum_{g\in G} \chi_i(g) \overline{\chi_j(g)}
  = \sum_{k=1}^h \frac{|\cC_k|}{|G|} \chi_i(\cC_k) \overline{\chi_j(\cC_k)}.
$$
Set $D(G)$ to be the diagonal matrix with diagonal entries $(p_1,\dots,p_h)$,
where $p_k = |\cC_k|/|G|$ is the weight of the conjugacy class $\cC_k$.
The orthogonality relations are then expressed by the equality
$$
I_h = A(G) D(G) A(G)^\dagger,
$$
where $\dagger$ denotes the conjugate transpose. The matrix $D(G)$ can
be viewed as the inner product matrix of the Haar measure induced by $G$
on $\Cl(G)$.

\subsubsection*{Rational character table.}
Let $\chi$ be a character on $G$, let $m$ be the exponent of $G$, and
let $\cC = \cC(g)$ be a conjugacy class. As the trace of a representation
of $g$, the value $\chi(\cC)$ lies in $\ZZ[\zeta_m]$, since each of its
eigenvalues are in $\bbmu_m = \langle\zeta_m\rangle$.  We thus obtain two
actions of the Galois group $\Gal(\QQ(\zeta_m)/\QQ) \isom (\ZZ/m\ZZ)^*$.
Denote $\sigma: (\ZZ/m\ZZ)^* \rightarrow \Gal(\QQ(\zeta_m)/\QQ)$ the
isomorphism such that $\zeta_m^{\sigma(k)} = \zeta_m^k$. The first of
the actions is on conjugacy classes, by $\cC(g) \mapsto \cC(g^k)$, and
the second on characters by $\chi^{\sigma(k)}(\cC(g)) = \chi(\cC(g))^{\sigma(k)}$.
Considering the action on eigenvalues we see immediately that
$$
\chi^{\sigma(k)}(\cC(g)) = \chi(\cC(g^k)).
$$

\subsubsection*{Restriction from $\cR(\Sym{n})$.}
Only characters in the image of $\cR(\Sym{n})$ can be parametrized by polynomials
in $\ZZ[s_1,\dots,s_{n-1}]$ from the standard representation.
We note by example, that the pre-image of $\cC$ in $\Cl(\Sym{n})$ under the
induced map $\Cl(G) \rightarrow \Cl(\Sym{n})$ can split into an even number
of conjugacy class separated by a quadratic character not coming from $\Sym{n}$.
We observe this phenomenon for $G = D_4$ and $G = Q_8$ in the examples section
below.

\section{Algorithms for Galois representations}

In what follows we describe algorithms for testing equivalence of finite
Galois characters.  As principal application, we consider input $f(x)$
of degree $n$, determining a number field $K = \QQ[x]/(f(x))$, and
describe how to evaluate a sample set of primes $S$ at characters on
the permutation group $\cG(K)$. The approach is completely general,
allowing one to compare the set of characters on the absolute group $\cG$
mapping through permutation groups $\cG(K_1)$ and $\cG(K_2)$ determined
by number fields $K_1$ and $K_2$.

\subsection*{Factorization types of irreducible polynomials}

Consider an irreducible polynomial $f(x)$ in $\ZZ[x]$ of degree $n$,
set $K = \QQ[x]/(f(x))$ and let $L$ be its normal closure with
maximal order $\cO_L$. For a rational prime $p$ and prime $\fP$
over $p$ in $\cO_L$, the Frobenius lift $\Frob_\fP$ is the unique
element of the decomposition subgroup $D_\fP \subset G = \cG(K)$, such that
$$
\Frob_\fP(a) \equiv a^p \bmod \fP
$$
for all $a$ in $\cO_L$.
Denote by $\Frob_p$ the conjugacy class of $\Frob_\fP$ in $\Cl(G)$.

For $p$ not dividing $\disc(f(x))$ we define the {\it factorization type}
of $f(x) \bmod p$ to be the multiset of degrees of the factorization
of $f(x)$ in $\FF_p[x]$, which we may denote $(d_1,\dots,d_t)$, where
$d_1 \le \cdots \le d_t$ and $d_1 + \cdots + d_t = n$.  We can now identify
the data of the factorization type with the cycle type of the Galois group
$G = \cG(K)$ equipped with its embedding in $\Sym{n}$.

\begin{lemma}
The factorization type of $f(x) \bmod p$ is the cycle type of $\Frob_p \subset \cG(K)$.
\end{lemma}

\begin{proof}
The factorization $p\cO_K = \pp_1 \cdots \pp_t$ is determined from $f(x) \equiv f_1(x) \cdots f_t(x)
\bmod p$, with $\pp_k = (p,f_k(x))$ a prime of degree $d_k = \deg(f_k)$. The Galois group acts
transitively on primes of $\cO_L$ over $p$, and there exist conjugates $\fP_1,\dots,\fP_t$ over
$\pp_1,\dots,\pp_t$, from which we see that $d_k$ divides $\deg(\fP)$, and each $d_k$ is the
cardinality of an orbit of roots modulo $p$ under the action of $\Frob_\fP$.
\end{proof}

\subsection*{Character inner products as expectation.}

The factorization type of a polynomial gives a means of taking random samples
of character values $(s_1,\dots,s_{n-1})$ at a set $S$ of primes mapping to
the group $G$. Other data, for particular characters, may come from weight one
modular eigenforms, character sums, or Kronecker symbols.
Let $S$ be such a sample set of primes, and $\psi$, $\chi$
two characters which can be evaluated on $S$. We write $\psi(p)$ and $\chi(p)$
for the value of the character at a sample point.
We obtain an approximation for the orthogonal product $\langle\psi,\chi\rangle$
as the expectation of $\psi\overline{\chi}$:
$$
\langle\psi,\chi\rangle = \EE(\psi\overline{\chi})
  \sim \EE_S(\psi\overline{\chi}) = \frac{1}{|S|}\sum_{p\in S} \psi(p)\overline{\chi}(p).
$$
Assuming the multiplicity of each irreducible character in the support of
$\psi$ and $\chi$ is one, then $m = \langle\psi,\chi\rangle$ is an integer
counting the number of irreducible characters in the support of both
$\psi$ and $\chi$.  When $\psi$ and $\chi$ are irreducible, to determine
equality $\psi = \chi$, one needs only sufficient precision to distinguish
the one bit $\langle\psi,\chi\rangle = 0$ or $\langle\psi,\chi\rangle = 1$.

The interest in working with irreducible characters, or nearly irreducible
characters as captured by the image of restriction from $\cR(\Sym{n})$, is that
the variance of the character products $\psi\overline{\chi}$ is minimized,
and the number of primes needed to recognize convergence small, as observed
by Shieh~\cite{Shieh15,Shieh16} in the case of symplectic groups $\USp(2m)$
(see also Fit\'e and Guitart~\cite{Fite18}).

One should note that in view of classifying the Galois group, nonvanishing of an
element of the kernel ideal of the restriction $\cR(\Sym{n}) \rightarrow \cR(G)$
can be used to  provably exclude $G$ as a Galois group.
This was already observed by Pohst~\cite{Pohst96}, who proposed the use of
factorization types as a lower bound for the Galois group, and that for $n \ge 8$
the factorization types, and their probabilities, fail to separate groups.
This statement, however, concerns the data of the induced Haar measure on $\Cl(G)$,
and not that of the character table of $G$.
Precisely we have two data structures on $\Cl(G)$ at our disposal, that of
a probability space and of class functions (given by a character table):

\begin{itemize}
\item
$\Cl(G)$ with Haar measure $p:\Cl(G)\rightarrow\RR$, and
\item
$\CC^h = \Hom(\Cl(G),\CC)$ with orthonormal basis $\{\chi_1,\dots,\chi_h\}$.
\end{itemize}
Due to failure of surjectivity of the restriction homomorphism from $\cR(\Sym{n})$,
the subset of characters determined from the cycle types are unlikely to separate
groups for sufficiently large $n$.  Nevertheless, the joint data of Haar measure
and character table, plus the system of restriction maps coming from common
embeddings in $\Sym{n}$ gives more information than either the Haar measure or
character table alone.

\subsection*{Restriction kernel ideal}

To a conjugacy class $\cC$ for $\Sym{n}$ we associate an ideal $\mm_\cC$ in
$\ZZ[s_1,\dots,s_{n-1}]$ of the form
$$
\mm_\cC = (s_1-s_1(\cC),\dots,s_{n-1}-s_{n-1}(\cC)),
$$
where $(s_1(\cC),\dots,s_{n-1}(\cC))$ are the values of $s_i$ at $\cC$. Then
the kernel ideal for the restriction of $\cR(\Sym{n})$ to $\cR(G)$ is the
intersection ideal
$$
I(G) = \bigcap_{\cC\in\pi(\Cl(G))} \mm_{\cC},
$$
where $\pi: \Cl(G) \rightarrow \Cl(\Sym{n})$.

\subsubsection*{Example.}
Consider the restriction from $\cR(O(3))$ to $\cR(\Sym{4})$. Since
$$
\cR(O(3)) = \frac{\ZZ[s_1,s_2,s_3]}{(s_1s_3-s_2,s_3^2-1)}
$$
and the values of $(s_1,s_2,s_3)$ are in
$$
\{(3,3,1),\ (-1,-1,1),\ (0,0,1),\ (-1,1,-1),\ (1,-1,-1)\},
$$
we obtain a defining ideal of $\Sym{4}$ given by the additional generators:
$$
s_1(s_1+1)(s_1-s_3-2),\ s_1(s_1+1)(s_1-1)(s_1-3),\ (s_1+1)(s_1-1)(s_3-1).
$$
The map $\Cl(D_4) \rightarrow \Cl(\Sym{4})$ fails to surject on $(0,0,1)$, hence
there are only four maximal ideals in the intersection and the kernel ideal for
$\cR(\Sym{4}) \rightarrow \cR(D_4)$ is generated by:
$$
s_1^2 - s_1 - s_2 - s_3 - 2,\ s_2 - s_1 s_3,\ s_3^2 - 1.
$$
The first polynomial is not in the kernel ideal for $\cR(\Sym{4})$ and its
vanishing provides a test for $D_4$.  Geometrically, it means that the 
tensor square of the representation with trace $s_1$ decomposes into a
direct sum of representations with trace $s_1 + s_2 + s_3 + 2$.

\subsection*{Restriction homomorphism}

Let $H \subset G$ be permutation groups, and set $\ell = |\Cl(H)|$ and $h = |\Cl(G)|$
their cardinalities of their conjugacy class sets.
Suppose that $\{\psi_1,\dots,\psi_\ell\}$ and $\{\chi_1,\dots,\chi_h\}$ are the
irreducible characters, which are given by embeddings in $\CC^\ell$ and $\CC^h$,
respectively. We thus have isomorphisms
$$
\cR(H) = \bigoplus_{i=1}^\ell \ZZ\psi_i \longrightarrow \Lambda(H) \subset \CC^\ell,
\mbox{ and }
\cR(G) = \bigoplus_{j=1}^h \ZZ\chi_j \longrightarrow \Lambda(G) \subset \CC^h,
$$
where $\Lambda(H)$ and $\Lambda(G)$ are the lattices in $\CC^\ell$ and $\CC^h$
spanned by the rows of the character table.
The restriction homomorphism $\cR(G) \mapsto \cR(H)$ is induced by the map
$\pi:\Cl(H) \rightarrow \Cl(G)$, by
$$
\chi \longmapsto (\chi(\pi(\cC_1)),\dots,\chi(\pi(\cC_\ell))) \in \Lambda(H) \subset \CC^\ell.
$$
The linear transformation $\Lambda(G) \rightarrow \Lambda(H)$ gives the restriction
homomorphism as an integral $(h\times\ell)$-matrix with respect to the respective
bases of irreducible characters.  The rows of this matrix can be interpretted as
{\it branching rules}, giving the decomposition of an irreducible character on $G$
as a sum of irreducible characters on $H$.

Inside each $\Lambda(G)$ we have a sublattice (generally of lower rank) $\Lambda_\QQ(G)
= \Lambda(G) \cap \QQ^h$ of rational-valued characters.  We recall that for a
conjugacy class $\cC$ of group elements of order $m$, the value of $\chi(\cC)$ is a
sum of eigenvalues in $\QQ(\zeta_m)$.  We thus obtain an action by the Galois group
of a cyclotomic field on the irreducible characters.  As a consequence, the lattice
$\Lambda_\QQ(G)$ is generated by the sums over Galois orbits of irreducible characters.
Since these orbits are disjoint, this basis of rational characters remains orthogonal,
but not orthonormal, since $\langle\chi,\chi\rangle$ measures the cardinality of
the orbit (assuming $\chi$ is a sum of irreducible characters of multiplicity one).
On the other hand, the restriction images $\Res^G_H(\Lambda(G)) \subset \Lambda(H)$ and
$\Res^G_H(\Lambda_\QQ(G)) \subset \Lambda_\QQ(H)$ do not possess natural reduced
orthogonal bases.  In order to determine a generating set which is small with
respect to the orthogonality relations on characters, we need to apply a constrained
lattice reduction inside the submonoid of characters:
$$
\bigoplus_{j=1}^\ell \NN\psi_j \subset \bigoplus_{j=1}^\ell \ZZ\psi_j = \cR(H).
$$
Rather than a generic LLL algorithm, we need to carry out a structured lattice
reduction in the character monoid order to be able to invoke the heuristic
arguments for convergence of small characters.

\subsection*{Algebraic parametrization}
In order to interpret factorization types of polynomials (or splitting types of primes)
as conjugacy classes on which we can apply the class functions $s_1,\dots,s_{n-1}$,
we need to find an explicit algebraic parametrization
$$
\frac{\ZZ[s_1,\dots,s_{n-1}]}{I(\Sym{n})} \rightarrow \cR(\Sym{n})
\rightarrow \Res^{\Sym{n}}_G(\Lambda(\Sym{n})) \subseteq \Lambda(G)
$$
The presentation $\ZZ[s_1,\dots,s_{n-1}]/I(\Sym{n}) \rightarrow \cR(\Sym{n})$ comes
from the standard representation of $\Sym{n}$, and its composition into $\Lambda(\Sym{n})$
can be effectively computed. In order to lift characters in $\Lambda(\Sym{n})$ back to
representative polynomials in $(s_1,\dots,s_{n-1})$, we must invert
$$
\frac{\ZZ[s_1,\dots,s_{n-1}]}{I(\Sym{n})} \rightarrow \Lambda(\Sym{n}).
$$
As noted above, the isomorphism $\cR(\Sym{n}) \rightarrow \Lambda(\Sym{n})$ is obtained
by the Chinese remainder theorem.  More precisely, over $\QQ$, we obtain a product
decomposition of the \'etale algebra $\cR(\Sym{n}) \otimes \QQ$:
$$
\cR(\Sym{n})\otimes\QQ \longrightarrow
  \frac{\cR(\Sym{n})}{\mm_{\cC_1}} \otimes\QQ\times \cdots \times
  \frac{\cR(\Sym{n})}{\mm_{\cC_h}} \otimes\QQ \isom \QQ^h.
$$
under which $\cR(\Sym{n}) \isom \Lambda(\Sym{n}) \subseteq \ZZ^h$.  Since the generators
$s_1,\dots,s_{n-1}$ can be evaluated at conjugacy classes, we can evaluate a basis of
monomials modulo $I(\Sym{n})$ and invert a matrix to determine the pre-image
of a basis of irreducible characters.  The same applies to a basis of characters
in $\Res^{\Sym{n}}_G(\Lambda(\Sym{n}))$ modulo the restriction kernel $I(G)$.

\subsection*{Database of restriction-induction}

Databases of transitive permutation groups of degree up to 30 are available in
GAP~\cite{GAPGroup} and Magma~\cite{MagmaArticle,MagmaGroup}, computed by
Greg Butler, John McKay, Gordon Royle and Alexander Hulpke (see \cite{ButlerMcKay83},
\cite{Butler93}, \cite{Royle87}, \cite{CHM98}, \cite{Hulpke05}).
The above is intended to motivate an interest in a metastructure of the restriction
relations (and adjoint induction relations) between character rings $\cR(G)$, and for the
algebraic parametrizations arizing from the restriction homomorphism from orthogonal
groups.

\section{Explicit computations}

We illustrate the approach through arithmetic statistics of character theory by applying
the methods to groups of low degree.  First we analyze the dihedral and quaternionic
groups $D_4$ and $Q_8$ of order $8$, the smallest groups sharing the same character table.
Then we consider an example of a pair of permutation groups of degree $8$ and order $16$
whose cycle types and induced Haar measure on $\Sym{8}$-conjugacy classes are equal.
We show how an auxillary (sub)field suffices to distinguish the characters using joint
Frobenius cycle data.  In a final example, we treat different permutation representations
of $\Alt{5}$, to show how this approach can be used to establish the equivalence of the
absolute Galois representations determined by different fields.
\vspace{2mm}

\subsection*{Dihedral and quaternionic groups of order $8$}
The groups $D_4$ and $Q_8$, known to share the same character table, can nevertheless
be separated by the restriction data coming from a permutation representation.
We first recall that the common character table takes the form
$$
A(G) = \left[\begin{array}{@{\;}r@{\;}r@{\;}r@{\;}r@{\;}r@{\;}}
1 &  1 &  1 &  1 &  1\\
1 &  1 & -1 &  1 & -1\\
1 &  1 &  1 & -1 & -1\\
1 &  1 & -1 & -1 &  1\\
2 & -2 &  0 &  0 &  0
\end{array}\right]
$$
with weights $(1/8,1/8,1/4,1/4,1/4)$ on the conjugacy classes.
\ignore{
The group algebra $\QQ[D_4]$ decomposes as
$$
\QQ[D_4] \isom \QQ \times \QQ \times \QQ \times \QQ \times M_2(\QQ),
$$
while the group algebra $\QQ[Q_8]$ decomposes as
$$
\QQ[Q_8] \isom \QQ \times \QQ \times \QQ \times \QQ \times \HH,
$$
where $\HH$ is the quaternion algebra over $\QQ$ ramified at $2$ and $\infty$.
}
The semisemimple group algebras $\QQ[D_4]$ and $\QQ[Q_8]$ have Wedderburn
decompositions
$$
\QQ[D_4] \isom \QQ \times \QQ \times \QQ \times \QQ \times M_2(\QQ),
\mbox{ and }
\QQ[Q_8] \isom \QQ \times \QQ \times \QQ \times \QQ \times \HH,
$$
where $\HH$ is the quaternion algebra over $\QQ$ ramified at $2$ and $\infty$.
These decompositions correspond to the four linear characters and sole degree-2
irreducible representation.

Only the former group, $D_4$, embeds in $\Sym{4}$, which shows that the permutation
embedding contains distinguishing information not in the character table.
We make explicit the above approach through character theory for the degree-$4$
permutation representation.  Let $\{1,\chi_1,\chi_2,\chi_3,\chi_4\}$ be
a basis of characters, with $\chi_1$, $\chi_2$, and $\chi_3 = \chi_1\chi_2$
quadratic linear characters, and $\chi_4$ of degree $2$.  The standard
representation of $\Sym{4}$ in $\rO(3)$ provides irreducible characters
$$
\{1,s_1,s_2,s_3,s_1^2 - s_1 - s_2 - 1\}
$$
where $s_3$ is the quadratic determinant character, $s_1$ and $s_2 = s_1s_3$
are degree-$3$ representations, and the last one is of degree $2$.
Computing the inner product matrices for these characters on $\Sym{4}$ and $D_4$,
we obtain
$$
\left[\begin{array}{@{\;}c@{\;\;}c@{\;\;}c@{\;\;}c@{\;\;}c@{\;}}
1 & 0 & 0 & 0 & 0\\
0 & 1 & 0 & 0 & 0\\
0 & 0 & 1 & 0 & 0\\
0 & 0 & 0 & 1 & 0\\
0 & 0 & 0 & 0 & 1
\end{array}\right]
\mbox{ and }
\left[\begin{array}{@{\;}c@{\;\;}c@{\;\;}c@{\;\;}c@{\;\;}c@{\;}}
1 & 0 & 0 & 0 & 1\\
0 & 2 & 1 & 0 & 0\\
0 & 1 & 2 & 0 & 0\\
0 & 0 & 0 & 1 & 1\\
1 & 0 & 0 & 1 & 2
\end{array}\right]\!\cdot
$$
For example, this was the output to the nearest integer for the expectation method
on a sample size of 16 unramified primes, for the polynomials $x^4 + x + 1$ and
$x^4 - 2x^2 + 2$ with respective Galois groups $\Sym{4}$ and $D_4$.

One identifies the polynomial expression $\chi = s_1^2 - s_1 - s_2 - 1$ for the
irreducible degree-$2$ character $\chi$ on $\Sym{4}$, which decomposes into a direct
sum $1 + s_3$ on $D_4$, from which we deduce that $s_1^2 - s_1 - s_2 - s_3 - 2$
is in the kernel ideal $I(D_4)$. Similarly, we read from the inner products
$\langle{s_1,s_1}\rangle = \langle{s_2,s_2}\rangle = 2$ and $\langle{s_1,s_2}\rangle = 1$
on $D_4$ that each of $s_1$ and $s_2$ decompose into two irreducible
characters, which share a common irreducible summand.  The restriction homomorphism
from $\cR(\Sym{4})$ thus captures
$$
1, s_1 = \chi_1 + \chi_4, s_2 = \chi_2 + \chi_4, s_3 = \chi_3.
$$
The restriction fails to span all characters, because the conjugacy classes are not
separated by characters on $\Sym{4}$. Indeed the cycle types of the five conjugacy
classes in $\Cl(D_4)$ are $1^4$, $1^2 2^1$, $2^2$, $2^2$, and $4^1$, and hence
the two classes of cycle type $2^2$ map to the same class in $\Cl(\Sym{4})$.

The missing character $\chi_1$ is easily recovered.  It arises from the quadratic
subfield (here with defining polynomial $x^2 - 2x + 2$), which can be expressed as
a Legendre symbol.  In terms of the basis of characters $\{1,s_1,s_2,s_3,\chi_1\}$,
we now obtain an inner product matrix:
$$
\left[
\begin{array}{cccc|c}
1 & 0 & 0 & 0 & 0 \\
0 & 1 & 0 & 0 & 0 \\
0 & 0 & 2 & 1 & 1 \\
0 & 0 & 1 & 2 & 0 \\ \hline
0 & 0 & 1 & 0 & 1
\end{array}\right]
$$
which can be reduced to an orthonormal basis for $\cR(D_4)$.
\vspace{1mm}

Since both $D_4$ and $Q_8$ admit permutation representations of degree $8$, we carry
out a similar analysis of the permutation representations of degree $8$ for $D_4$
and $Q_8$, given by
$$
D_4 \isom \left\langle
    \begin{array}{@{}l@{}}
    (1, 8)(2, 7)(3, 4)(5, 6),\\
    (1, 2)(3, 5)(4, 6)(7, 8),\\
    (1, 6)(2, 4)(3, 8)(5, 7)\end{array}\right\rangle
\mbox{ and }
Q_8 \isom \left\langle
    \begin{array}{@{}l@{}}
    (1,2,4,7)(3,6,8,5),\\
    (1,3,4,8)(2,5,7,6)
    \end{array}\right\rangle\cdot
$$
The cycle types $1^8$, $2^4$, $4^2$ arise with probabilities $(1/8,5/8,1/4)$ in $D_4$ whereas
in $Q_8$, theses same types have probabilities $(1/8,1/8,3/4)$.
Both groups embed in $\Alt{8} \subset \SO(7)$, hence the character rings are parametrized
by $\cR(\SO(7)) \isom \ZZ[s_1,s_2,s_3]$ ($s_7 = 1$ and $s_4 = s_3$, $s_5 = s_2$, $s_6 = s_1$).
Since the cycle types are the same, the kernel ideals agree, but the Haar measures differentiate
the groups. However, a naive tabulation of the probabilities gives a poor empirical invariant.
In fact, computing these probabilities is tantamount to evaluating the expectations of the
idempotents $e_1$, $e_2$, $e_3$ under the isomorphism
$$
\cR(G) \otimes \QQ = \frac{\QQ[s_1,s_2,s_3]}{I(G) \otimes \QQ} \longrightarrow
  \frac{\cR(G) \otimes \QQ}{\mm_{\cC_1} \otimes \QQ} \times
  \frac{\cR(G) \otimes \QQ}{\mm_{\cC_2} \otimes \QQ} \times
  \frac{\cR(G) \otimes \QQ}{\mm_{\cC_3} \otimes \QQ} \isom \QQ \times \QQ \times \QQ.
$$

To express this computation in the character ring framework, we scale by the group order to
have integer values.  As a general strategy for a group $G \subset \Sym{n}$ this amounts to
asking whether the scaled idempotents converge to
$$
( \langle|G|e_1,1\rangle,\dots\langle|G|e_s,1\rangle) = (|\cC_1|,\dots,|\cC_s|),
$$
where $\cC_i$ are the $\Sym{n}$-conjugacy classes for $G$.

Let $\{1,\chi_1,\chi_2,\chi_3,\psi\}$ be a basis of irreducible characters for $D_4$, and
$\{1,\chi_1',\chi_2',\chi_3',\psi'\}$ be a basis of irreducible characters for $Q_8$.
The parametrization gives a $\QQ$-basis $\{1,s_1,s_2\}$ and an idempotent basis
$\{e_1,e_2,e_3\}$ which are characteristic functions for the evaluations on conjugacy classes.
A reduced basis for the image of $\cR(\Sym{8})$ in $\cR(D_4)$ is $\{1,\sigma_1,\sigma_2\}$,
described as follows in these respective bases:
$$
\begin{array}{cccc}
D_4      & \{1,s_1,s_2\}      & \{e_1,e_2,e_3\}  & \{1,\chi_1,\chi_2,\chi_3,\psi\} \\ \hline
1        & 1                  & e_1 + e_2 + e_3  & 1\\
\sigma_1 & -s_1 + s_2/2 - 1/2 & 2e_1 - e_2 + e_3 & \chi_1 + \psi\\
\sigma_2 & 2s_1 - s_2/2 + 1/2 & 4e_1 - 2e_3      & \chi_1 + \chi_2 + \chi_3
\end{array}
$$
Similarly, a reduced basis for the image of $\cR(\Sym{8})$ in $\cR(Q_8)$ is $\{1,\tau_1,\tau_2\}$,
expressed in the respective bases as follows:
$$
\begin{array}{cccc}
Q_8    & \{1,s_1,s_2\}      & \{e_1,e_2,e_3\}   & \{1,\chi_1',\chi_2',\chi_3',\psi'\} \\ \hline
1      & 1                  & e_1 + e_2 + e_3   & 1\\
\tau_1 & -s_1 + s_2/2 - 3/2 & 2e_1 - 2e_2       & \psi'\\
\tau_2 & 3s_1 - s_2 + 3     & 3e_1 + 3e_2 - e_3 & \chi_1' + \chi_2' + \chi_3'
\end{array}
$$
Relative to the parametrizations from $\cR(\SO(7))$, the bases $(\sigma_1,\sigma_2)$
and $(\tau_1,\tau_2)$ are related by
$(\sigma_1,\sigma_2) = (\tau_1+1,\tau_1+\tau_2-1)$, and inversely
$(\tau_1,\tau_2) = (\sigma_1-1,\sigma_1+\sigma_2+2)$. we thus express $(8e_1,8e_2,8e_3$
in the respective bases:
$$
\begin{array}{cccc}
       & \{1,s_1,s_2\}   & \{1,\sigma_1,\sigma_2\}   & \{1,\tau_1,\tau_2\}\\ \hline
8e_1   & s_1 + 1         & 1 + \sigma_1 + \sigma_2   & 1 + 2\tau_1 + \tau_2\\
8e_2   & 5s_1 - 2s_2 + 7 & 5 - 3\sigma_1 + \sigma_2  & 1 - 2\tau_1 + \tau_2\\
8e_3   & -6s_1 + 2s_2    & 2 + 2\sigma_1 - 2\sigma_2 & 6 - 2\tau_2
\end{array}
$$
giving inclusions of submodules
$
\langle{8e_1,8e_2,8e_3}\rangle \subset
\langle{1,\sigma_2,\sigma_3}\rangle =
\langle{1,\tau_2,\tau_3}\rangle \subset
\langle{e_1,e_2,e_3}\rangle$.

Computing the expectations of the test functions $\{1,\sigma_1,\sigma_2\}$,
for $D_4$ on polynomials with Galois groups $G = D_4$ or $Q_8$,
the Gram matrix $M(G) = (\EE(\sigma_i\sigma_j))$ ($\sigma_0 = 1$) takes the form
$$
M(G) = \left[
\begin{array}{@{\;}c@{\;\;}c@{\;\;}c@{\;}} 1 & 0 & 0\\ 0 & 2 & 1\\ 0 & 1 & 3\end{array}\right]
\mbox{ where } G = D_4 \mbox{ and otherwise }
\left[
\begin{array}{@{}r@{\;\;\;}r@{\,}r@{\;}} 1 & 1 &-1\\ 1 & 2 & 0\\-1 & 0 & 5\end{array}\right]\!\!\cdot
$$
With respect to test functions $\{1,\tau_1,\tau_2\}$ for $Q_8$, the Gram matrices are
$$
M(G) = \left[
\begin{array}{@{\;}c@{\;\;}c@{\;\;}c@{\;}} 1 & 0 & 0\\ 0 & 1 & 0\\ 0 & 0 & 3\end{array}\right]
\mbox{ where } G = Q_8 \mbox{ and otherwise }
\left[
\begin{array}{@{}r@{\,}r@{\,}r@{\;}} 1 &-1 & 2\\-1 & 3 &-3\\ 2 &-3 & 7\end{array}\right]\!\!\cdot
$$
It should be clear that the full Gram matrix gives a more complete picture of the orthogonality
relations of charactes than the triple of inner products
$(\langle{8e_1,1}\rangle),(\langle{8e_2,1}\rangle),(\langle{8e_3,1}\rangle)$,
which is just one linear combination of the rows in the above Gram matrices.

In the next section, we show that the choice of reduced basis for the target group gives
a better set of test functions, converging more rapidly to the asymptotic Gram matrix.
With respect to the polynomials $x^8 + 6x^4 + 1$ of Galois group $D_4$ and
$x^8 - 12x^6 + 36x^4 - 36x^2 + 9$ of Galois group $Q_8$, we obtain reasonably good
convergence (to within a half integer) with the first 80 primes.

\subsection*{Non distinguished representations of degree $8$}
The first example of nonisomorphic permutation representations not distinguished
by their cycle types and Haar measure are the degree-8 groups of order $16$
denoted {\tt 8T10} and {\tt 8T11} (see the LMFDB~\cite{LMFDB} Galois groups
database). Specifically we define the representative groups
$$
\begin{array}{l}
G_0 = \langle{(1,2,3,8)(4,5,6,7), (1,5)(3,7)}\rangle\mbox{ and }\\ 
G_1 = \langle{(1,3,5,7)(2,4,6,8), (1,4,5,8)(2,3,6,7), (1,5)(3,7)}\rangle
\end{array}
$$
whose character tables are given by
$$
A(G_0) =
\left[\begin{array}{@{\;}*{10}{r@{\,}}r@{\;}}
1&  1&  1&  1&  1&  1&  1&  1&  1&  1\\
1&  1&  1&  1& -1& -1& -1& -1&  1&  1\\
1&  1&  1&  1& -1& -1&  1&  1& -1& -1\\
1&  1&  1&  1&  1&  1& -1& -1& -1& -1\\
1& -1& -1&  1&  1& -1&  i& -i& -i&  i\\
1& -1& -1&  1& -1&  1& -i&  i& -i&  i\\
1& -1& -1&  1& -1&  1&  i& -i&  i& -i\\
1& -1& -1&  1&  1& -1& -i&  i&  i& -i\\
2& -2&  2& -2&  0&  0&  0&  0&  0&  0\\
2&  2& -2& -2&  0&  0&  0&  0&  0&  0
\end{array}\right]
\mbox{ and }
A(G_1) =
\left[\begin{array}{@{\;}*{10}{r@{\,}}r@{\;}}
1&  1&   1&  1&  1&  1&  1&  1&  1&  1\\
1&  1&   1&  1& -1& -1&  1& -1& -1&  1\\
1&  1&  -1& -1& -1& -1& -1&  1&  1&  1\\
1&  1&  -1& -1&  1&  1& -1& -1& -1&  1\\
1&  1&   1&  1&  1& -1& -1&  1& -1& -1\\
1&  1&   1&  1& -1&  1& -1& -1&  1& -1\\
1&  1&  -1& -1& -1&  1&  1&  1& -1& -1\\
1&  1&  -1& -1&  1& -1&  1& -1&  1& -1\\
2& -2& -2i& 2i&  0&  0&  0&  0&  0&  0\\
2& -2&  2i&-2i&  0&  0&  0&  0&  0&  0
\end{array}\right]
$$
\ignore{
K<i> := QuadraticField(-1);
D := DiagonalMatrix(K,[1/16,1/16,1/16,1/16,1/8,1/8,1/8,1/8,1/8,1/8]);
A := Matrix([
    [K| 1,  1,  1,  1, 1, 1,  1,  1,  1,  1],
    [K| 1,  1,  1,  1,-1,-1,  1, -1, -1,  1],
    [K| 1,  1, -1, -1,-1,-1, -1,  1,  1,  1],
    [K| 1,  1, -1, -1, 1, 1, -1, -1, -1,  1],
    [K| 1,  1,  1,  1, 1,-1, -1,  1, -1, -1],
    [K| 1,  1,  1,  1,-1, 1, -1, -1,  1, -1],
    [K| 1,  1, -1, -1,-1, 1,  1,  1, -1, -1],
    [K| 1,  1, -1, -1, 1,-1,  1, -1,  1, -1],
    [K| 2, -2,-2*i,2*i,0, 0,  0,  0,  0,  0],
    [K| 2, -2,2*i,-2*i,0, 0,  0,  0,  0,  0] ]);
Abar := Parent(A)![ Trace(c)-c : c in Eltseq(A) ];
A*D*Transpose(Abar);}%
with respective probabilities $(1/16,1/16,1/16,1/16,1/8,1/8,1/8,1/8,1/8,1/8)$.
We note that the first 8 characters are linear, and the latter two are of
degree $2$.  The linear characters admit a group structure, isomorphic to
$C_2 \times C_4$ and $C_2^3$, respectively.  We denote the characters by
$
\{1,\chi_1,\chi_2,\chi_3,\rho_1,\bar{\rho}_1,\rho_2,\bar{\rho}_2,\psi_1,\psi_2\}
\mbox{ and }
\{1,\xi_1,\xi_2,\xi_3,\xi_4,\xi_5,\xi_6,\xi_7,\psi,\bar{\psi}\}.
$
In the groups $G_0$ and $G_1$ the character of the standard representation
(of degree 7) decomposes as
$$
s_1 = \chi_1 + \rho_1 + \bar{\rho}_1 + \psi_1 + \psi_2
\mbox{ and }
s_1 = \xi_1 + \xi_2 + \xi_2 + \psi + \bar{\psi},
$$
respectively, but the individual characters in $s_1$ are not separated.

Given the obvious Galois action (on the codomain field $\QQ(i)$), we see that
the subrings $\cR_\QQ(G_0)$ and $\cR_\QQ(G_1)$ have different ranks, $8$ and~$9$.
On the other hand, the images of the restriction homomorphism from $\cR(\Sym{8})$
have rank $4$ in each of $\cR(G_0)$ and $\cR(G_1)$, generated for instance by
$\{1,s_1,s_2,s_3\}$.
Moreover, since exactly the same four cycle types occur, with the same
probabilities $(1/16, 5/16, 1/8, 1/2)$, the characters in the image of restriction
from $\cR(\Sym{8})$ to $\cR(G_0)$ and $\cR(G_1)$ can not be differentiated.

Let $K_0$ and $K_1$ be number fields whose normal closures have respective Galois
groups $G_0$ and $G_1$.
In order to distinguish these fields, it suffices to construct missing characters
from the linear character groups.  In fact these number fields have nontrivial
automorphism groups, isomorphic to $V_4$ and $C_4$, respectively.  This induces
respective subfield lattices of the forms
$$
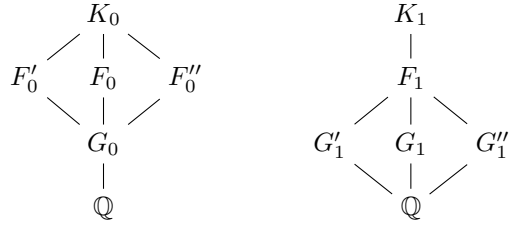
\begin{tikzcd}[cramped,sep=small]
    & K_0 &     & & &    & K_1 &     \\
F_0' \arrow[ur,dash] & F_0 \arrow[u,dash] & F_0'' \arrow[ul,dash] & & &   & F_1 \arrow[u,dash] &     \\
    & G_0 \arrow[ul,dash] \arrow[u,dash] \arrow[ur,dash] & & &   & G_1' \arrow[ur,dash] & G_1 \arrow[u,dash] & G_1'' \arrow[ul,dash]\\
    & \QQ \arrow[u,dash] &     & & &   & \QQ \arrow[ul,dash] \arrow[u,dash] \arrow[ur,dash]
\end{tikzcd}
$$
For each field we recover a significant subgroup of the linear character groups
from the quartic and quadratic characters.  In fact there is a unique
cyclic subfield $F_0/\QQ$ in $K_0$ which recovers the characters
$\rho_1$, $\bar{\rho}_2$, and $\chi_1 = \rho_1^2$. (The other fields $F_0'$ and
$F_0''$ are non-normal.)  And there exists a unique biquadratic field $F_1/\QQ$
in $K_1$ which yields the quadratic characters $\xi_1$, $\xi_2$, $\xi_3$.
The pairs $(K_0,F_0)$ and $(K_1,F_1)$ give characters on the pairs of permutation
groups  of degree $8$ and~$4$, $(G_0,G_0/H_0 \isom C_4)$ and $(G_1,G_1/H_1 \isom V_4)$,
such that the joint factorization types of Frobenius characters separate
the Galois structures.

\subsection*{Representations of $\Alt{5}$}
We denote the irreducible characters of the alternating groups $\Alt{5}$
by $\{1,\chi_1,\chi_2,\chi_3,\chi_4\}$, where $\chi_1$ is the character of the degree-$4$
standard representation, $\chi_2$ is the character of a degree-$5$ representation, and $\chi_3$
and $\chi_4$ are the conjugate characters of degree-$3$ icosohedral representations over
$\QQ(\sqrt{5})$. The rational representations are thus spanned by the orthogonal characters
$\{1,\chi_1,\chi_2,\chi_3+\chi_4\}$ of degrees $1$, $4$, $5$, and $6$.

On the other hand, the permutation representation of $\Alt{5}$ in $\Sym{5}$ gives a
parametrization by
$$
\cR(\SO(4)) = \frac{\ZZ[s_1,s_2,s_3,s_4]}{(s_1-s_3,s_4-1)} \isom \ZZ[s_1,s_2],
$$
and while $|\Cl(\Alt{5})| = 5$, there are two conjugacy classes which map to the same
cycle type $5^1$ in $\Cl(\Sym{5})$.  Thus the restriction from $\cR(\Sym{5})$ gives a
basis of four independent characters, and we identify:
$$
(1,s_1,s_1^2 - s_2 - s_1 - 1,s_2) = (1,\chi_1,\chi_2,\chi_3+\chi_4).
$$
In addition to its degree-$5$ permutation representation, $\Alt{5}$ admits a faithful
permutation representation in $\Sym{6}$.
In the restriction of $\ZZ[s_1,s_2] \isom \cR(\SO(5))$ we recognize the same characters
equipped with a different parametrization
$$
(1, s_1^2 - 2 s_1 - s_2 - 1, s_1, s_2 - \chi_1) = (1, \chi_1, \chi_2, \chi_3 + \chi_4).
$$

Consider the number fields, each with Galois group $\Alt{5}$, defined by polynomials
$$
\begin{array}{l}
  f = x^5 - 5 x^4 + 48 x^3 + 28 x^2 + 5 x - 1,\\
  g = x^{6} + 4 x^{5} + 10 x^4 - 10 x^3 + 17 x^2 + 10 x + 1
\end{array}
$$
constructed as subfields of the same normal closure. Although not isomorphic,
we can construct the inner product matrix of the same characters set
$\{1,\chi_1,\chi_2,\chi_3+\chi_4\}$ on $\Alt{5}$ with respect to its
different embeddings in $\Sym{5}$ and $\Sym{6}$.
Jointly evaluating the characters on factorization types of $f$ or $g$
with those of either $f$ or $g$, yields the same diagonal inner product
matrix ($= \mathrm{diag}(1,1,1,2)$ to nearest integer).
This gives a means of recognizing the same character of the absolute Galois group
via different presentations. The arithmetic statistic approach through character
theory gives a powerful tool to not only characterize Galois groups, but to recognize
equivalence of finite representations of the absolute Galois group $\cG$ which
may arise in different contexts.

\section{Variance, covariance and convergence}

The focus on irreducible characters provides, on the one hand, a theoretic framework
for understanding the arithmetic statistics of Frobenius distributions.
On the computational side, irreducible characters provide test functions with optimal
convergence properties.
Naively, the orthogonality relations for a system $\{\chi_1,\dots,\chi_r\}$ of
irreducible characters as test functions, it suffices to recognize the integer
$\langle\chi_i,\chi_j\rangle = \delta_{ij}$ to one bit of precision.
Furthermore, $\chi_i \ne 1$ and $\chi_j \ne 1$ the inner products
$
\langle\chi_i,1\rangle = \langle\chi_j,1\rangle = 0
$
imply that $\chi_i$ and $\chi_j$ have mean $0$, hence we can interpret
$$
\EE_S(\chi_i\overline{\chi}_j) = \frac{1}{|S|}\sum_{p\in S} \chi_i(p)\overline{\chi}_j(p)
$$
as a (sample) variance ($i=j$) or covariance ($i\ne j$) of the sample $S$, we
see that the use of irreducible characters (or of reduced characters in $\cR(G)$
as the next best approximation when irreducible characters are not in the
restriction image from $\cR(\Sym{n})$) minimizes the variance of the test
functions, and orthogonality minimizes the covariance.

We can illustrate the convergence properties with the lattice of subgroups between
the representation of $\PSL_2(\FF_7)$ on $\PP^1(\FF_7)$ and $\Sym{8}$:
$$
\begin{tikzcd}
\PGL_2(\FF_7) \isom G_1                               \arrow[rr,hook] &                    & \Sym{8}\\
\PSL_2(\FF_7) \isom H_1 \arrow[u,shift right=2.5em,hook] \arrow[r,hook]  & H_2 \arrow[r,hook] & \Alt{8} \arrow[u,hook]
\end{tikzcd}
$$
with respective orders $|H_1| = 168$, $|G_1| = 336$, and $|H_2| = 1344$.

Let $h(G)$ be the number of conjugacy classes of $G$, equal to the number of irreducible characters
and to the rank of $\cR(G)$;
let $r(G)$ be the number of characters irreducible over $\QQ$, equal to the rank of $\cR_\QQ(G)$;
and let $s(G)$ the rank of the image of the restriction of $\cR(\Sym{n})$ to $\cR(G)$.
For each of the groups we give the respective numbers $h(G)$, $r(G)$ and $s(G)$, as well as
a representative polynomial (from the LMFDB~\cite{LMFDB}) with Galois group $G$.
$$
\begin{array}{ccccl}
  G     & h(G) & r(G) & s(G) & f_G(x)\\ \hline
\Sym{8} &  22  &  22  &  22  & x^8 - x - 1\\
\Alt{8} &  14  &  12  &  12  & x^8 - 2x^7 + 3x^5 - 5x^4 + 2x^3 + 2x^2 - x + 1\\
  G_1   &   9  &   8  &   8  & x^8 - x^7 + x^6 + 4x^5 - x^4 - 3x^3 + 5x^2 - 2x + 1\\
  H_2   &  11  &  10  &   8  & x^8 - 4x^7 + 8x^6 - 9x^5 + 7x^4 - 4x^3 + 2x^2 + 1\\
  H_1   &   6  &   5  &   5  & x^8 - 4x^7 + 7x^6 - 7x^5 + 7x^4 - 7x^3 + 7x^2 + 5x + 1\\
\end{array}
$$
For the generic group $\Sym{n}$ the characters $(1,s_1,\dots,s_{n-1})$ are irreducible
on $\Sym{n}$ and form a system of test functions for $\Sym{n}$.  On $\Alt{n}$ and its
subgroups the relations $s_{n-1-i} = s_i$ hold, and so the characters $(1,s_1,\dots,s_m)$,
where $n = 2m+1$ or $2m+2$, form a system of test functions for $\Alt{n}$.

The Gram matrices $M(G)$ with respect to the test characters $(1,s_1,\dots,s_7)$
for $G = \Sym{8}$, $\Alt{8}$, and $G_1$, respectively are:
$$
M(\Sym{8}) =
\left[\begin{array}{@{\,}*{7}{c@{\;\;}}c@{\,}}
1 & 0 & 0 & 0 & 0 & 0 & 0 & 0\\
0 & 1 & 0 & 0 & 0 & 0 & 0 & 0\\
0 & 0 & 1 & 0 & 0 & 0 & 0 & 0\\
0 & 0 & 0 & 1 & 0 & 0 & 0 & 0\\
0 & 0 & 0 & 0 & 1 & 0 & 0 & 0\\
0 & 0 & 0 & 0 & 0 & 1 & 0 & 0\\
0 & 0 & 0 & 0 & 0 & 0 & 1 & 0\\
0 & 0 & 0 & 0 & 0 & 0 & 0 & 1
\end{array}\right]\!\!\ccomma\;
M(\Alt{8}) =
\left[\begin{array}{@{\,}*{7}{c@{\;\;}}c@{\,}}
1 & 0 & 0 & 0 & 0 & 0 & 0 & 1\\
0 & 1 & 0 & 0 & 0 & 0 & 1 & 0\\
0 & 0 & 1 & 0 & 0 & 1 & 0 & 0\\
0 & 0 & 0 & 1 & 1 & 0 & 0 & 0\\
0 & 0 & 0 & 1 & 1 & 0 & 0 & 0\\
0 & 0 & 1 & 0 & 0 & 1 & 0 & 0\\
0 & 1 & 0 & 0 & 0 & 0 & 1 & 0\\
1 & 0 & 0 & 0 & 0 & 0 & 0 & 1
\end{array}\right]\!\!\ccomma\;
M(G_1) =
\left[\begin{array}{@{\,}*{7}{c@{\;\;}}c@{\,}}
1 & 0 & 0 & 0 & 1 & 0 & 0 & 0\\
0 & 1 & 0 & 1 & 1 & 1 & 0 & 0\\
0 & 0 & 3 & 2 & 1 & 1 & 1 & 0\\
0 & 1 & 2 & 6 & 4 & 1 & 1 & 1\\
1 & 1 & 1 & 4 & 6 & 2 & 1 & 0\\
0 & 1 & 1 & 1 & 2 & 3 & 0 & 0\\
0 & 0 & 1 & 1 & 1 & 0 & 1 & 0\\
0 & 0 & 0 & 1 & 0 & 0 & 0 & 1
\end{array}\right]\!\!\cdot
$$
For the indicated representative polynomials, characters $(\chi_1,\dots,\chi_r)$ 
and set of non-ramified primes $S$, we define the error matrix:
$
Z_S(G) = \EE_S(\chi_i\overline{\chi}_j) - M(G)
$
and for an $(r\times r)$-matrix $Z = (z_{ij})$ and define the normalized $\ell_p$-norms
$$
||Z||_p = \Big(\frac{1}{r^2}\sum_{i,j} |z_{ij}|^p\Big)^{1/p} \mbox{ and } ||Z||_\infty = \max_{i,j}\{ |z_{ij}| \}.
$$
In particular we need $||Z_S(G)||_\infty < 0.50$ in order for the approximation to round
to $M(G)$.  We say that a sequence stably converges to $M(G)$ after $m$ terms if
$||Z_S(G)||_\infty < 0.50$ for all initial segments $S$ of the sequence with $|S| > m$.

Setting $S$ equal to the first $128k$ non-ramified primes, in the case of $\Sym{8}$ and $\Alt{8}$
the symmetric functions give good convergence in the $\ell_2$, $\ell_8$ and $\ell_\infty$-norms
to $M(G)$ on small sample sets consisting of the first $128k$ non-ramified primes.
$$
\begin{array}{r@{\;}*{5}{c@{\,}}}
   & \multicolumn{5}{c}{\Sym{8}} \\
   & ||Z_S(G)||_2 & & ||Z_S(G)||_8 & & ||Z_S(G)||_\infty \\ \hline
 1:& 0.104870 & < & 0.184799 & < & 0.257812\\
 2:& 0.104915 & < & 0.197659 & < & 0.269531\\
 3:& 0.093747 & < & 0.189553 & < & 0.255208\\
 4:& 0.072267 & < & 0.138632 & < & 0.191406\\
 5:& 0.063890 & < & 0.112834 & < & 0.151562\\
 6:& 0.063620 & < & 0.115167 & < & 0.171875\\
 7:& 0.052897 & < & 0.083975 & < & 0.116071\\
 8:& 0.045921 & < & 0.070367 & < & 0.097656\\
\end{array}
\quad
\begin{array}{r@{\;}*{5}{c@{\,}}}
   & \multicolumn{5}{c}{\Alt{8}} \\
   & ||Z_S(G)||_2 & & ||Z_S(G)||_8 & & ||Z_S(G)||_\infty \\ \hline
 1:& 0.080624 & < & 0.112569 & < & 0.140625\\
 2:& 0.099134 & < & 0.174740 & < & 0.226562\\
 3:& 0.074997 & < & 0.128586 & < & 0.166666\\
 4:& 0.057739 & < & 0.092246 & < & 0.119140\\
 5:& 0.058826 & < & 0.128167 & < & 0.181250\\
 6:& 0.053728 & < & 0.112338 & < & 0.158854\\
 7:& 0.049278 & < & 0.098191 & < & 0.138392\\
 8:& 0.036335 & < & 0.065900 & < & 0.092773\\
\end{array}
$$
Even with sample size 128, we obtain a close approximation to the correct
Gram matrix, and the convergence remains stable.  In contrast, for the
group $G_1$ (of index 120 in $\Sym{8}$) taking increments of size $1024$
we find that $2^{14} = 1024 \cdot 16$ primes gives an exact approximation
of $M(G_1)$ (in the $\ell_\infty$-norm) but that at least $1024 \cdot 22$
primes are needed for stable convergence:
$$
\begin{array}{r@{\;}*{5}{c@{\,}}}
   & \multicolumn{5}{c}{G_1} \\
   & ||Z_S(G)||_2 & & ||Z_S(G)||_8 & & ||Z_S(G)||_\infty \\ \hline
 1:& 0.876885 & < & 1.841975 & < & 2.686523\\
 2:& 0.229706 & < & 0.475835 & < & 0.701171\\
 3:& 0.437539 & < & 0.862551 & < & 1.233723\\
 4:& 0.542897 & < & 1.080542 & < & 1.525878\\
 5:& 0.267850 & < & 0.528893 & < & 0.756054\\
 6:& 0.365931 & < & 0.733534 & < & 1.035156\\
 7:& 0.199105 & < & 0.407255 & < & 0.580217\\
 8:& 0.229675 & < & 0.471416 & < & 0.672363\\
 9:& 0.111158 & < & 0.231270 & < & 0.333224\\
\end{array}
\quad
\begin{array}{r@{\;}*{5}{c@{\,}}}
   & \multicolumn{5}{c}{G_1} \\
   & ||Z_S(G)||_2 & & ||Z_S(G)||_8 & & ||Z_S(G)||_\infty \\ \hline
10:& 0.187304 & < & 0.375945 & < & 0.533105\\
11:& 0.211544 & < & 0.429012 & < & 0.613725\\
12:& 0.231261 & < & 0.465137 & < & 0.665364\\
13:& 0.279154 & < & 0.560439 & < & 0.800030\\
14:& 0.201504 & < & 0.399819 & < & 0.572195\\
15:& 0.189139 & < & 0.375454 & < & 0.534960\\
16:& 0.178182 & < & 0.348732 & < & 0.493652\\
17:& 0.143345 & < & 0.282338 & < & 0.397633\\
18:& 0.136637 & < & 0.266879 & < & 0.378417\\
\end{array}
$$
Extending the computation further, we find that the apparent stable convergence fails
when $||Z_S(G_1)||_\infty > 0.50$ for $|S| = 1024\cdot k$ for $19 \le k \le 21$ and
again in the range $45 \le k \le 48$.

Passing to a basis of rational irreducible characters ($r(G_1) = s(G_1)$), the
rational character table $A(G_1)$ and the inner product matrix $D(G_1)$ of the Haar
measure on conjugacy classes are respectively
$$
A(G_1) =
\left[\begin{array}{@{}c@{\;}*{6}{r@{\;}}r@{\;}}
 1 &  1 &  1 &  1 &  1 &  1 &  1 &  1\\
 1 &  1 & -1 &  1 &  1 & -1 &  1 & -1\\
 6 & -2 &  0 &  0 &  2 &  0 & -1 &  0\\
12 &  4 &  0 &  0 &  0 &  0 & -2 &  0\\
 7 & -1 &  1 &  1 & -1 &  1 &  0 & -1\\
 7 & -1 & -1 &  1 & -1 & -1 &  0 &  1\\
 8 &  0 & -2 & -1 &  0 &  1 &  1 &  0\\
 8 &  0 &  2 & -1 &  0 & -1 &  1 &  0
\end{array}\right] \mbox{ and }
D(G_1) =
\frac{1}{336}
\left[\begin{array}{@{\;}*{7}{c@{\;\;}}c@{\;}}
1 &  0 &  0 &  0 &  0 &  0 &  0 &  0\\
0 & 21 &  0 &  0 &  0 &  0 &  0 &  0\\
0 &  0 & 28 &  0 &  0 &  0 &  0 &  0\\
0 &  0 &  0 & 56 &  0 &  0 &  0 &  0\\
0 &  0 &  0 &  0 & 42 &  0 &  0 &  0\\
0 &  0 &  0 &  0 &  0 & 56 &  0 &  0\\
0 &  0 &  0 &  0 &  0 &  0 & 48 &  0\\
0 &  0 &  0 &  0 &  0 &  0 &  0 & 84
\end{array}\right]\!\!\ccomma
$$
which determine the diagonalized matrix 
$M(G_1) = A(G_1) D(G_1) A(G_1)^t = \mathrm{diag}(1,1,1,2,1,1,1,1)$
with respect to the rational irreducible characters.
With respect to this basis, in increments of $128k$ primes, we find stable
convergence after just $512 = 128\cdot 4$ primes:
\vspace{-1mm}
$$
\begin{array}{r@{\;}*{5}{c@{\,}}}
   & \multicolumn{5}{c}{G_1} \\
   & ||Z_S(G)||_2 & & ||Z_S(G)||_8 & & ||Z_S(G)||_\infty \\ \hline
 1:& 0.191903 & < & 0.557482 & < & 0.937500\\
 2:& 0.107457 & < & 0.204107 & < & 0.312500\\
 3:& 0.111166 & < & 0.316320 & < & 0.531250\\
 4:& 0.085609 & < & 0.199992 & < & 0.335937\\
 5:& 0.087717 & < & 0.208395 & < & 0.350000\\
 6:& 0.094278 & < & 0.217121 & < & 0.364583\\
 7:& 0.103194 & < & 0.236602 & < & 0.397321\\
 8:& 0.110885 & < & 0.249000 & < & 0.417968\\
\end{array}
\quad
\begin{array}{r@{\;}*{5}{c@{\,}}}
   & \multicolumn{5}{c}{G_1} \\
   & ||Z_S(G)||_2 & & ||Z_S(G)||_8 & & ||Z_S(G)||_\infty \\ \hline
 9:& 0.114006 & < & 0.234514 & < & 0.392361\\
10:& 0.116967 & < & 0.233938 & < & 0.390625\\
11:& 0.120169 & < & 0.241507 & < & 0.403409\\
12:& 0.090920 & < & 0.197313 & < & 0.330729\\
13:& 0.093108 & < & 0.180276 & < & 0.300480\\
14:& 0.070129 & < & 0.145311 & < & 0.243303\\
15:& 0.074861 & < & 0.160193 & < & 0.268750\\
16:& 0.030534 & < & 0.066387 & < & 0.111328
\end{array}
$$
For the subgroup chain $H_1 \subset H_2 \subset \Alt{8}$, starting with the characters
$(1,s_1,s_2,s_3)$, irreducible on $\Alt{8}$, we find a similar analysis.
In particular, the Gram matrices with respect to this basis are
$$
M(\Alt{8}) =
\left[\begin{array}{@{\,}*{3}{c@{\;\;}}c@{\,}}
1 & 0 & 0 & 0\\
0 & 1 & 0 & 0\\
0 & 0 & 1 & 0\\
0 & 0 & 0 & 1
\end{array}\right]\!\!\ccomma\;\;
M(H_2) =
\left[\begin{array}{@{\,}*{3}{c@{\;\;}}c@{\,}}
1 & 0 & 0 & 0\\
0 & 1 & 0 & 0\\
0 & 0 & 1 & 0\\
0 & 0 & 0 & 3
\end{array}\right]\!\!\ccomma\;\;
M(H_1) =
\left[\begin{array}{@{\,}*{3}{c@{\;\;}}c@{\,}}
1 & 0 & 0 & 1\\
0 & 1 & 1 & 2\\
0 & 1 & 4 & 3\\
1 & 2 & 3 & 10
\end{array}\right]\!\!\cdot
$$
In the former two cases, the characters are orthogonal and irreducible
or nearly so ($s_3$ decomposes as a sum of three distinct irreducibles
on $H_2$), and convergence is relatively good.
In constrast, the Gram matrix $M(H_1)$ has determinant $14$, and far from
being orthogonal or irreducible (except for $1$ and $s_1$) on $H_1$.
In increments of $1024$, we find stable convergence only after $2^{15}
= 1024 \cdot 32$ primes:
\vspace{-3mm}
$$
\begin{array}{r@{\;}*{5}{c@{\,}}}
   & \multicolumn{5}{c}{H_1} \\
   & ||Z_S(G)||_2 & & ||Z_S(G)||_8 & & ||Z_S(G)||_\infty \\ \hline
 1:& 1.300776 & < & 2.685076 & < & 3.787109\\
 2:& 0.457035 & < & 0.943691 & < & 1.331054\\
 3:& 0.316304 & < & 0.671333 & < & 0.948242\\
 4:& 0.149549 & < & 0.327977 & < & 0.463623\\
\longversion{%
\multicolumn{1}{c}{\vdots} & \vdots &   & \vdots   &   & \vdots \\}{%
 5:& 0.209291 & < & 0.449120 & < & 0.634570\\
 6:& 0.250692 & < & 0.529368 & < & 0.747558\\
 7:& 0.281838 & < & 0.585979 & < & 0.826869\\
 8:& 0.354998 & < & 0.731739 & < & 1.031982\\
 9:& 0.271338 & < & 0.559605 & < & 0.789279\\
10:& 0.375442 & < & 0.771937 & < & 1.088476\\
11:& 0.232194 & < & 0.486329 & < & 0.686523\\
12:& 0.112442 & < & 0.240716 & < & 0.340087\\}
13:& 0.201940 & < & 0.417409 & < & 0.588792\\
14:& 0.219831 & < & 0.449876 & < & 0.634137\\
15:& 0.207462 & < & 0.427431 & < & 0.602799\\
16:& 0.170705 & < & 0.352046 & < & 0.496520
\end{array}
\quad
\begin{array}{r@{\;}*{5}{c@{\,}}}
   & \multicolumn{5}{c}{H_1} \\
   & ||Z_S(G)||_2 & & ||Z_S(G)||_8 & & ||Z_S(G)||_\infty \\ \hline
17:& 0.162082 & < & 0.331846 & < & 0.467773\\
18:& 0.249476 & < & 0.507743 & < & 0.715332\\
19:& 0.260497 & < & 0.533048 & < & 0.751336\\
20:& 0.250136 & < & 0.514311 & < & 0.725195\\
\longversion{%
\multicolumn{1}{c}{\vdots} & \vdots &   & \vdots   &   & \vdots \\}{%
21:& 0.200098 & < & 0.413150 & < & 0.582728\\
22:& 0.211578 & < & 0.432938 & < & 0.610218\\
23:& 0.203890 & < & 0.413032 & < & 0.581606\\
24:& 0.162417 & < & 0.325936 & < & 0.458455\\
25:& 0.124521 & < & 0.245254 & < & 0.343945\\
26:& 0.153018 & < & 0.302437 & < & 0.424429\\
27:& 0.210064 & < & 0.418420 & < & 0.587962\\
28:& 0.188783 & < & 0.372945 & < & 0.523367\\}
29:& 0.183952 & < & 0.364100 & < & 0.511112\\
30:& 0.193960 & < & 0.384122 & < & 0.539257\\
31:& 0.148770 & < & 0.290129 & < & 0.406060\\
32:& 0.132390 & < & 0.258615 & < & 0.362091
\end{array}
$$
Going further one finds that the $\ell_\infty$-norm gradually decreases and does indeed
stay below $0.50$ after this point.
In contrast, in terms of the basis $(1,\chi_1=\varphi+\bar\varphi,\chi_2,\chi_3,\chi_4)$
of irreducible characters over $\QQ$, of degrees $(1,6,6,7,8)$ given by
$$
\begin{array}{ll}
\chi_1 = (4 s_2 + 3 s_3 - s_1 s_2 - 4 s_1 - 2)/2, & \chi_3 = s_1,\\
\chi_2 = (2 s_2 + 5 s_3 - s_1 s_2 - 6 s_1 - 4)/4, & \chi_4 = (s_1 s_2 + 2 s_1 + 2 - 2 s_2 - 3 s_3)/2,
\end{array}
$$
the test characters stable converge to $M(H_1)$ after only $128$ primes, with
results here in increments of $128$ primes:
\vspace{-3mm}
$$
\begin{array}{r@{\;}*{5}{c@{\,}}}
   & \multicolumn{5}{c}{H_1} \\
   & ||Z_S(G)||_2 & & ||Z_S(G)||_8 & & ||Z_S(G)||_\infty \\ \hline
 1:& 0.227868 & < & 0.301886 & < & 0.406250\\
 2:& 0.225747 & < & 0.296604 & < & 0.398437\\
 3:& 0.127307 & < & 0.165588 & < & 0.216145\\
 4:& 0.149822 & < & 0.191605 & < & 0.250000\\
 5:& 0.166819 & < & 0.214155 & < & 0.271875\\
 6:& 0.085019 & < & 0.114926 & < & 0.148437\\
 7:& 0.101179 & < & 0.132950 & < & 0.166294\\
 8:& 0.114860 & < & 0.148922 & < & 0.193359
\end{array}
\quad
\begin{array}{r@{\;}*{5}{c@{\,}}}
   & \multicolumn{5}{c}{H_1} \\
   & ||Z_S(G)||_2 & & ||Z_S(G)||_8 & & ||Z_S(G)||_\infty \\ \hline
 9:& 0.064413 & < & 0.088651 & < & 0.114583\\
10:& 0.079219 & < & 0.104501 & < & 0.132812\\
11:& 0.091419 & < & 0.119029 & < & 0.154829\\
12:& 0.056475 & < & 0.076844 & < & 0.097656\\
13:& 0.047871 & < & 0.066901 & < & 0.086538\\
14:& 0.041653 & < & 0.062817 & < & 0.083705\\
15:& 0.029993 & < & 0.041051 & < & 0.053125\\
16:& 0.041465 & < & 0.054989 & < & 0.069335
\end{array}
$$
These convergence results give empirical support to the principle of using
irreducible characters as test functions, based on the theoretical
interpretation of inner product relations on characters as variance and covariance.
Moreover, when using irreducible characters, the number of primes necessary
to recognize the Gram matrix associated to a Galois group is strikingly small.

\section{Asymptotics in the degree}

In analyzing the character theory of a permutation group of large degree,
one must avoid certain bottlenecks in the complexity. First the number of
transitive permutation groups is too large to enumerate, and so clearly
the poset must be navigated in a lazy fashion.
Second, the number of conjugacy classes (hence of irreducible characters)
for $\Sym{n}$ is too large to enumerate. For the generic groups $\Sym{n}$
and $\Alt{n}$, the characters $(1,s_1,\dots,s_{n-1})$ and $(1,s_1,\dots,s_m)$,
where $n = 2m+1$ or $2m+2$, give a subset of rational irreducible test
functions (when $n = 2m+2$, the character $s_m$ is the sum of two characters
on $\Alt{n}$, conjugate over a quadratic field).
In general the number of conjugacy classes is the partition number $p(n)$,
whose asymptotic growth is known by Hardy and Ramanujan~\cite{HR1918} to be
\vspace{-2mm}
$$
p(n) \sim \frac{1}{4n\sqrt{3}}\exp\left(\pi\sqrt{\frac{2n}{3}}\right)\!\cdot
$$
In particular, we will treat a nontrivial example of degree $120$ (and $240$)
despite the large size $p(120) = 1844349560$ (and $p(240) = 105882246722733$)
of the corresponding partition numbers.
Finally, computation of the kernel ideal of the restriction $\cR(O(n-1))
\rightarrow \cR(G)$ by Groebner basis algorithms is prohibitively expensive,
even if the $s(G)$ points in the kernel can be computed.

Polynomials with interesting Galois groups of large degree, outside the
generic groups $\Sym{n}$ and $\Alt{n}$ and cyclic and dihedral groups
$C_n$ and $D_n$ rely on specific constructions.
We consider such an example of Jouve, Kowalski and Zywina~\cite{JKZ2008},
a polynomial $f(x)$ of degree $240$ with Galois group the Weyl group $W(E_8)$
of the lattice $E_8$, of order $696729600$.  In contrast to the large number
of conjugacy classes of $\Sym{240}$, the number of conjugacy classes of
$W(E_8)$ is $112$, and the restriction homomorphism from $\cR(\Sym{240})$
has full rank.  We take the quotient of order $348364800$ by its center, which
is the Galois group of the degree $120$ polynomial $g(x)$ such that $f(x) = g(x^2)$.
The quotient group $G = W(E_8)/Z(W(E_8))$ has $67$ conjugacy classes, all
characters are rational, and the restriction homomorphism from $\cR(\Sym{120})$
is a subring of rank $65$. We consider the $18$ absolutely irreducible rational
characters in the image. In increments of $256$ primes, we compute the convergence
to the Gram matrix $A(G)$ for these $18$ characters to $2^{13} = 256 \cdot 32$
primes:
$$
\begin{array}{r@{\;}*{5}{c@{\,}}}
   & \multicolumn{5}{c}{W(E_8)/Z(W(E_8))} \\
   & ||Z_S(G)||_2 & & ||Z_S(G)||_8 & & ||Z_S(G)||_\infty \\ \hline
 1:& 0.512041 & < & 1.947691 & < & 3.843750\\
 2:& 0.256200 & < & 0.868283 & < & 1.609375\\
 3:& 0.180087 & < & 0.525172 & < & 0.929687\\
 4:& 0.251753 & < & 0.848164 & < & 1.571289\\
\longversion{%
\multicolumn{1}{c}{\vdots} & \vdots &   & \vdots   &   & \vdots \\}{%
 5:& 0.206391 & < & 0.654058 & < & 1.199218\\
 6:& 0.177971 & < & 0.537867 & < & 0.981770\\
 7:& 0.160677 & < & 0.453450 & < & 0.818080\\
 8:& 0.141711 & < & 0.379362 & < & 0.668945\\
 9:& 0.133769 & < & 0.330092 & < & 0.571180\\
10:& 0.126439 & < & 0.299904 & < & 0.495312\\
11:& 0.153668 & < & 0.422025 & < & 0.744318\\
12:& 0.147983 & < & 0.379336 & < & 0.666015\\}
13:& 0.161766 & < & 0.376064 & < & 0.648137\\
14:& 0.151537 & < & 0.350967 & < & 0.611049\\
15:& 0.139688 & < & 0.325884 & < & 0.550000\\
16:& 0.128557 & < & 0.298173 & < & 0.504638\\
\end{array}
\quad
\begin{array}{r@{\;}*{5}{c@{\,}}}
   & \multicolumn{5}{c}{W(E_8)/Z(W(E_8))} \\
   & ||Z_S(G)||_2 & & ||Z_S(G)||_8 & & ||Z_S(G)||_\infty \\ \hline
17:& 0.122505 & < & 0.279019 & < & 0.473345\\
18:& 0.118703 & < & 0.265146 & < & 0.452473\\
19:& 0.114018 & < & 0.254474 & < & 0.432360\\
20:& 0.110361 & < & 0.248728 & < & 0.442968\\
\longversion{%
\multicolumn{1}{c}{\vdots} & \vdots &   & \vdots   &   & \vdots \\}{%
21:& 0.124927 & < & 0.313813 & < & 0.584821\\
22:& 0.121894 & < & 0.304305 & < & 0.563742\\
23:& 0.119424 & < & 0.296752 & < & 0.546025\\
24:& 0.111043 & < & 0.265830 & < & 0.483561\\
25:& 0.108012 & < & 0.258921 & < & 0.462968\\
26:& 0.105207 & < & 0.252862 & < & 0.449969\\
27:& 0.117958 & < & 0.310375 & < & 0.580584\\
28:& 0.114853 & < & 0.299234 & < & 0.557756\\}
29:& 0.110513 & < & 0.283699 & < & 0.530980\\
30:& 0.108019 & < & 0.275711 & < & 0.514713\\
31:& 0.105191 & < & 0.262830 & < & 0.491053\\
32:& 0.102228 & < & 0.251891 & < & 0.468505
\end{array}
$$
Extending the computation further suggests that the convergence to $M(G)$ is stable
for $m > 2^{13}$.

\ignore{
Indeed:
$$
\begin{array}{r@{\;}*{5}{c@{\,}}}
   & \multicolumn{5}{c}{W(E_8)/Z(W(E_8))} \\
   & ||Z_S(G)||_2 & & ||Z_S(G)||_8 & & ||Z_S(G)||_\infty \\ \hline
33:& 0.100301 & < & 0.246619 & < & 0.456202\\
34:& 0.098624 & < & 0.241885 & < & 0.445312\\
35:& 0.097133 & < & 0.239099 & < & 0.437276\\
36:& 0.095606 & < & 0.235593 & < & 0.426432\\
37:& 0.095199 & < & 0.234092 & < & 0.420396\\
38:& 0.094019 & < & 0.232147 & < & 0.410053\\
39:& 0.092383 & < & 0.228793 & < & 0.403445\\
40:& 0.091207 & < & 0.227619 & < & 0.398730\\
41:& 0.090263 & < & 0.227170 & < & 0.403772\\
42:& 0.089389 & < & 0.226393 & < & 0.406436\\
43:& 0.086814 & < & 0.218609 & < & 0.377997\\
44:& 0.093970 & < & 0.247417 & < & 0.455255\\
45:& 0.092438 & < & 0.241997 & < & 0.442708\\
46:& 0.090766 & < & 0.235752 & < & 0.431216\\
47:& 0.089782 & < & 0.231993 & < & 0.421043\\
48:& 0.089241 & < & 0.231894 & < & 0.418294\\
\end{array}
$$
}

\section{Conclusion}

A standard tool in Galois group computation is to recognize the probable group from
an analysis of Frobenius cycle types.  We use an explicit polynomial parametrization
of the character ring to identify the irreducible characters in the restriction from
orthogonal groups and subsequently from the symmetric group.
As in the thesis work of Shieh~\cite{Shieh15,Shieh16}, with the view to classifying Sato-Tate
groups, it is recognized that the irreducible characters on the target group provide
optimal test functions for recognizing (or rejecting) a given group coming from a
Galois representation.
We develop this perspective in the application to the parametrized representation
rings of finite groups, with associated lattice structure.
Although we focus on Galois groups arising from splitting fields of polynomials over
$\QQ$, the same methods apply to Galois representations coming form $L$-series and
modular forms, families of exponential sums, and global fields of any characteristic.

At a higher level, the approach through character theory and arithmetic statistics
lets us identify when Frobenius distributions of different degrees admit a common
Galois subrepresentation.  Examples arise in the form of fields with isomorphic
normal closures, as described in the above examples of $\Alt{5}$ representations, but
more generally one can recognize whether two normal fields admit a common subfield.
In this framework orthogonality relations of characters are measured by correlations
of Frobenius distributions associated to different representations of the absolute
Galois group. This perspective has promising potential for the computational
investigation of Galois representations.

\subsection*{Acknowledgements.}
The author thanks Fernando Rodriguez--Villegas for suggesting the specialization to
finite groups of the author's work with Yih-Dar Shieh and Gilles Lachaud, on explicit
character theory of orthogonal groups, and for providing notes of his talk in Leiden
on the analysis of Galois groups by explicit character theory (see also the 2012
bachelor's thesis of van~Bommel~\cite{vanBommel}).
Thanks go also to Claus Fieker for discussions of his algorithm and code in Magma for
constructing subfields of the normal closure of a given number field, used in building
test examples of number fields with common Galois closure.
Finally, the author thanks an anonymous referee for suggesting the permutation groups
{\tt 8T10} and {\tt 8T11}, as the first example (of lowest degree) for which the cycle
distributions fail to distinguish the groups.  This work is dedicated to the memory
of Gilles Lachaud and our conversations on explicit character theory of compact Lie
groups.

\end{document}